\documentclass[a4paper]{amsart}
\usepackage[T1]{fontenc}
\usepackage[utf8]{inputenc}

\usepackage{amsmath}
\usepackage{amssymb}
\usepackage{amsthm}

\usepackage{mathpazo,courier}
\usepackage[scaled]{helvet}

\usepackage[numbers,sort&compress]{natbib}

\usepackage{enumerate}

\usepackage{tikz-cd}
\usepackage{hyperref}

\usepackage{leftidx}

\newcommand{\N}{\mathbb{N}}

\newcommand{\Q}{\mathbb{Q}}
\newcommand{\R}{\mathbb{R}}
\newcommand{\C}{\mathbb{C}}
\newcommand{\K}{\mathbb{K}}

\newcommand{\1}{\mathbb{1}}

\newcommand{\cO}{\mathcal{O}}

\newcommand{\p}{\mathfrak{p}}
\newcommand{\q}{\mathfrak{q}}

\newcommand\ph\varphi
\newcommand\ps\psi
\newcommand\ep\varepsilon
\newcommand\rh\varrho
\newcommand\al\alpha
\newcommand\be\beta
\newcommand\ga\gamma
\newcommand\om\omega
\newcommand\ta\tau
\renewcommand\th\vartheta
\newcommand\de\delta
\newcommand\ze\zeta
\newcommand\ch\chi
\newcommand\et\eta
\newcommand\io\iota
\newcommand\la\lambda
\newcommand\si\sigma

\newcommand\Ga\Gamma
\newcommand\De\Delta
\newcommand\Th\Theta
\newcommand\La\Lambda
\newcommand\Si\Sigma
\newcommand\Ph\Phi
\newcommand\Ps\Psi
\newcommand\Om\Omega

\DeclareMathOperator\Quot{Quot}

\DeclareMathOperator\End{End}

\DeclareMathOperator\im{im}

\DeclareMathOperator\Spec{Spec}

\DeclareMathOperator\Sym{Sym}
\DeclareMathOperator\Her{Her}
\DeclareMathOperator\U{U}
\let\O\undefined
\DeclareMathOperator\O{O}
\DeclareMathOperator\Mat{Mat}

\newtheorem{thm}{Theorem}[section]

\newtheorem*{thm*}{Theorem}
\newtheorem{prop}[thm]{Proposition}
\newtheorem{lem}[thm]{Lemma}
\newtheorem{cor}[thm]{Corollary}
\theoremstyle{definition}
\newtheorem{rem}[thm]{Remark}
\newtheorem*{quest*}{Question}

\newcommand{\fac}[3]{\Ph_{#3}({#1},{#2})}

\renewenvironment{proof}[1][\unskip]{\par\noindent {\em Proof #1: }}{{\qed\bigskip}}

\title{\normalsize Positive Semidefinite Univariate Matrix Polynomials}
\date{\small\today}
\author{Christoph Hanselka}
\address{Christoph Hanselka, The University of Auckland,Department of Mathematics, Private Bag 92019, Auckland 1142, New Zealand}
\email{c.hanselka@auckland.ac.nz}
\author{Rainer Sinn}
\address{Rainer Sinn, Max-Planck-Institut f\"ur Mathematik in den Naturwissenschaften, Inselstra\ss e 22, 04103 Leipzig, Germany}
\email{rsinn@mis.mpg.de}

\begin{document}
\begin{abstract}
	We study sum-of-squares representations of symmetric univariate real matrix polynomials that are positive semidefinite along the real line. We give a new proof of the fact that every positive semidefinite univariate matrix polynomial of size $n\times n$ can be written as a sum of squares $M=Q^TQ$, where $Q$ has size $(n+1)\times n$, which was recently proved by Blekherman-Plaumann-Sinn-Vinzant. Our new approach using the theory of quadratic forms allows us to prove the conjecture made by these authors that these minimal representations $M=Q^TQ$ are generically in one-to-one correspondence with the representations of the nonnegative univariate polynomial $\det(M)$ as sums of two squares.

	In parallel, we will use our methods to prove the more elementary hermitian analogue that every hermitian univariate matrix polynomial $M$ that is positive semidefinite along the real line, is a square, which is known as the matrix Fej\'er-Riesz Theorem.
\end{abstract}
\maketitle
\section*{Introduction}
A symmetric (or hermitian) matrix $M$ whose entries are polynomials with real (or complex) coefficients in $s$ variables $x_1,\dots,x_s$ is said to be positive semidefinite if the constant symmetric (or hermitian) matrix $M(x)$ is positive semidefinite for all $x\in\R^s$.

In this paper, we study sum-of-squares certificates for symmetric matrices whose entries are univariate real polynomials in $t$ to be positive semidefinite, that is factorizations of a univariate matrix polynomial $M(t)$ as a hermitian square, i.e.
\[
M(t) = Q(t)^*Q(t),
\]
where $Q(t)$ is a univariate $r\times n$ matrix polynomial for some integer $r$ and $Q(t)^*$ is the conjugate transpose of $Q(t)$ (complex conjugation is applied coefficient-wise to the entries of $Q$). Such a factorization makes it immediate that $M(t)$ is positive semidefinite at each $t\in\R$. This type of sum-of-squares certificates simultaneously generalizes the case of univariate polynomials, which we recover for $n=1$, and the case of constant matrices, which follows from the spectral theorem for matrices. So it is the simplest generalization of classical results in real algebraic geometry to the setup of matrix polynomials and a first step in gaining a better understanding of techniques that can be used to understand positive semidefinite matrix polynomials.

The existence of sum-of-squares certificates for hermitian univariate matrix polynomials was known since the Fifties at least and is often known as the matrix Fej\'er-Riesz Theorem, see \cite{WienerMR0097856}. This theorem has received much attention and has been generalized to various contexts in analysis, see \cite{DritschelMR2743422} for a recent survey.
Its mentioned matrix version directly implies the existence of sum-of-squares certificates for real symmetric $n\times n$ univariate matrix polynomials $M = Q^T Q$, where $Q$ has size $2n\times n$, which was proved by Choi-Lam-Reznick \cite{ChoiLamReznickMR566480}. 
Generalizations of this result to coefficient fields other than $\R$ have been developed in \citep{FRS06}.

In the symmetric case, the bound on the size of the matrix $Q(t)$ was later improved to $r=n+1$ by Blekherman-Plaumann-Sinn-Vinzant using techniques from projective algebraic geometry \cite{BPSV16} (and Leep \citep{Leep06} in an unpublished manuscript using techniques from the theory of quadratic forms). This bound $r = n+1$ is smallest possible and Blekherman-Plaumann-Sinn-Vinzant further observed that the number of essentially different sum-of-squares certificates $M = Q^TQ$ for a generic matrix polynomial $M$ and $r = n+1$ is finite and conjectured a count in terms of the degree of the determinant of $M$ as a univariate polynomial, \cite[Introduction]{BPSV16}.

In this paper, we prove this conjectured count by showing that factorizations $M = Q^TQ$ of a generic positive semidefinite real symmetric univariate $n\times n$ matrix polynomial $M$, where $Q$ is an $(n+1)\times n$ matrix polynomial, are in one-to-one correspondence with representations of the nonnegative univariate polynomial $\det(M)$ as a sum of two squares. 

\begin{thm*}[Corollary~\ref{cor:real_factorization}]
	Let $M\in\Sym_n(\R[t])$ be positive semidefinite with nonzero and square free determinant $\det(M)$. Then there is a bijection between the sets
	\[
		\{\, Q\in\Mat_{(n+1)\times n}(\R[t]) \mid Q^TQ=M\,\}
	\]
	and
	\[
		\{\, g\in\R[t]^2 \mid g^Tg=\det(M)\,\}
	\]
	modulo the left action of the orthogonal groups $\O_n(\R)$ and $\O_2(\R)$, respectively.
\end{thm*}

Another interesting consequence of our approach is a characterization of all real symmetric matrix polynomials $M$ that are squares, i.e.~that can be factored as $M = Q^TQ$, where the matrix polynomial $Q$ is of equal size $n\times n$.

\begin{thm*}[Corollary~\ref{cor:polynomial_factorization}]
	Let $M\in\Sym_n(\R[t])$ be positive semidefinite with nonzero determinant $\det(M)$. Then $M$ admits a square factorization, $M=Q^TQ$ for some $Q\in\Mat_n(\R[t])$, if and only if $\det(M)$ is a square in $\R[t]$.
\end{thm*}

We develop our theory mostly in parallel for both the real symmetric and complex hermitian setting. 
On the one hand, this leads to the following result, which is, as we learned during preparation of this paper, also due to \citep[Theorems 2,3]{ESS16}.
\begin{thm*}[Corollary~\ref{cor:complex_factorization}]
	Let $M\in\Her_n(\C[t])$ be positive semidefinite with nonzero and square free determinant $\det(M)$. Then the determinant map induces a bijection between the sets
	\[
		\{\, Q\in\Mat_n(\C[t]) \mid Q^*Q=M\,\}
	\]
	and
	\[
		\{\, g\in\C[t] \mid g^*g=\det(M)\,\}
	\]
	modulo the left action of the unitary groups $U_n(\C)$ and $U_1(\C)$, respectively.
\end{thm*}
On the other hand, this parallel approach highlights the essential differences between the real and the complex case. While the latter can be treated completely elementary, the former requires a considerable amount of additional work and quadratic forms theory. An indication that the complex hermitian case is indeed simpler is the fact that the determinant induces the bijection between factorization of $M$ and its determinant in Corollary~\ref{cor:complex_factorization}, whereas the bijection in the real symmetric case in Corollary~\ref{cor:real_factorization} is a lot more subtle and quite surprising (see Remark~\ref{rem:real_factorization_explicit}).

One of the central results for the proof of the above mentioned theorems is about Smith normal forms over the polynomial ring.
\begin{thm*}[Theorem~\ref{thm:factorization_snf}]
	Let $\K\in\{\R,\C\}$ and $M\in\Her_n(\K[t])$ be positive semidefinite with nonzero determinant. Then the equivalence classes of $n\times n$-factorizations $M=Q^*Q$ are in one-to-one correspondence to those of the monic Smith normal form of $M$. 
\end{thm*}
\subsection*{Reader's Guide}
There are two main technical ingredients for the proofs of our main results (Corollaries~\ref{cor:complex_factorization} and \ref{cor:real_factorization}), which are the following. Let $\K\in \{\R,\C\}$ and $M\in\Her_n(\K[t])$ be positive semidefinite.
\begin{itemize}
	\item We first show that factorizations $M = Q^*Q$ over $\K[t]$ are essentially the same as those over the localization $\cO$ at zeros of the determinant of $M$ (Theorem~\ref{thm:cassels_pfister}).
	\item Then we show that $M$ and its monic Smith normal form become congruent over $\cO$ (Theorem~\ref{thm:local_snf}).
\end{itemize}

The presentation is structured as follows. After fixing our conventions and recalling basic definitions in Section~\ref{sec:prelim}, we describe a special consequence of Witt's local-global principle to hermitian squares over the rational function field $\R(t)$.
In Section~\ref{sec:cassels_pfister} we study how hermitian factorizations over the rational function field can be turned into factorizations over the polynomial ring. 
We prove the existence of factorizations $M = Q^T Q$ of the generically smallest possible size $r = n+1$ and characterize those $M$ that admit factorizations of square size $r=n$.
The main result for the count of the number of smallest hermitian square representations is Theorem~\ref{thm:cassels_pfister}, which establishes a correspondence of factorizations over the polynomial ring $\R[t]$ and the rational function field $\R(t)$. This accomplishes the first of the two main steps described above. The second one is the main result of Section~\ref{sec:snf} in which we show that a positive semidefinite symmetric (or hermitian) matrix is congruent to its Smith normal form, if we allow certain denominators in the congruence transformations. An essential technical difficulty is to control these denominators. We then combine these two steps, in the following Section~\ref{sec:proofs} to prove the main results. In a short Appendix, Section~\ref{sec:appendix}, we discuss applications of the prime avoidance lemma to hermitian forms that we need earlier in the paper.

\subsection*{Acknowledgments}
We are grateful to Markus Schweighofer. Our approach extends fruitful discussions with him. The first author is supported by the Faculty Research Development Fund (FRDF) of The University of Auckland (project no. 3709120). The second author would like to thank Bernd Sturmfels and the Max-Planck-Institute in Leipzig for their hospitality and support.

\section{Preliminaries}\label{sec:prelim}
We recall some notions from linear algebra and fix our notation and terminology to avoid confusion. As a general reference for the theory of quadratic forms over rings, we refer to \citep{OMeara13}.
\begin{itemize}
	\item The polynomial rings in this paper will be over the field of real numbers or the field of complex numbers. Many statements will be developed in parallel for both cases, so we use $\K$ to mean $\R$ or $\C$. We want to note that throughout the paper, $\R$ and $\C$ can be replaced by any real closed field and its algebraic closure, respectively.
	\item The polynomial ring $\C[t]$ in one variable over the complex numbers is equipped with an involution (written as $\cdot^*$), given by coefficient wise complex conjugation and $t^* = t$.
	\item In the following, let $R$ be a commutative ring with involution written as $\cdot^*$. It might be the trivial one, as is the case for $\R[t]$.
	\item We write $\1_n$ for the $n\times n$ identity matrix.
	\item We call a square matrix $M\in\Mat_n(R)$ with entries in a ring $R$ \emph{non-degenerate} if $\det(M)$ is nonzero.

	\item The \emph{adjoint} of a matrix $A\in \Mat_{m\times n}(R)$, denoted by $A^*$, is the entry-wise conjugate of the transpose $A^T$.
	\item We denote the set of \emph{hermitian} $n\times n$ matrices over a ring $R$, i.e.~those $A\in\Mat_n(R)$ such that $A^*=A$, by $\Her_n(R)$ .
	\item We write $\U_n(R)=\{\, U \in \Mat_n(R)\mid U^*U=\1_n\,\}$ for the \emph{unitary group} over $R$.
	\item If the involution on the ring $R$ is trivial, then $\Her_n(R)=\Sym_n(R)$ is the set of \emph{symmetric matrices} and $\U_n(R)=\O_n(R)$ is the \emph{orthogonal group}.
	\item Given two square matrices $A\in\Mat_n(R)$ and $B\in\Mat_m(R)$ we denote by $A\oplus B\in\Mat_{n+m}(R)$ the block diagonal matrix with blocks $A$ and $B$.
	\item We write $\langle a_1,\dots,a_n\rangle$ for the diagonal matrix with diagonal entries $a_1,\dots,a_n$.
	\item We call $M,N\in\Her_n(R)$ \emph{congruent over the ring $R$}, written as $M\simeq_R N$, if there exists an invertible matrix $S\in \Mat_n(R)$ over $R$ such that $M=S^*NS$.
	\item
		For $M\in\Her_n(R)$ and $k\in\N$, we denote by $\fac MRk$ the set of equivalence classes (orbits) of $k\times n$-factorizations of $Q$
		\[
			\fac MRk := \{\, Q\in \Mat_{k\times n}(R) \mid Q^*Q=M\,\}/\U_k(R)
		\]
		where unitary group $\U_k(R)$ acts on the left.
	\item For a principal ideal domain $R$ and a non-degenerate matrix $M\in\Mat_n(R)$ with determinant $d:=\det(M)$ we define
		\begin{align*}
			\cO_M:&=\left\{\, \frac{a}b\in\Quot(R) \mid b\text{ is coprime to $d$}\,\right\}\\
			&=\bigcap_{\p\in Z(d)}R_\p
		\end{align*}
		where $Z(d)$ is the set of prime ideals containing $d$.  
		$\cO_M$ is a semi-local principal ideal domain and will play a central role in what follows.
\end{itemize}

\subsection{Quadratic forms over the rational function field $\R(t)$}
One central piece of our argument over the real numbers relies on Witt's local-global principle, which states that every totally indefinite quadratic form of dimension at least three over a function field of transcendence degree one over $\R$ is isotropic, i.e.~represents zero non-trivially. A proof can be found in \citep[Theorem 3.4.11]{Prestel_Delzell01}. Essential for the present paper is the following consequence, which is well-known in the quadratic forms community.
\begin{cor}\label{cor:witt_lgp}
	Let $a,b\in\R(t)$ be nonzero and positive semidefinite. Then $\langle a,b\rangle$ represents $1$ over $\R(t)$. In particular,
	\[
		\langle a,b\rangle\simeq_{\R(t)} \langle1,ab\rangle.
	\]
\end{cor}
\begin{proof}
	Consider the totally indefinite form $\langle a,b,-1\rangle$ and apply Witt's local-global principle to get a representation nontrivial representation
	\[
		ax^2+by^2-z^2=0.
	\]
	Since $a$ and $b$ are positive semidefinite, $z$ must be nonzero. Dividing by $z$ and adding $1$ to the equation we get a representation of $1$ by the form $\langle a, b\rangle$. An appropriate base change thus yields $\langle a,b\rangle\simeq \langle 1,c\rangle$ for some $c\in \R(t)$. Comparing determinants we get that $ab$ and $c$ differ by a square. Rescaling the second basis vector, we may assume that $ab=c$.
\end{proof}

Applying this corollary inductively one can easily characterize those non-degenerate $M\in\Sym_n(\R(t))$ that admit a square factorization $M=Q^TQ$ over the rational function field.
\begin{cor}\label{cor:rational_factorization}
	Let $M\in\Sym_n(\R(t))$ be non-degenerate and positive semidefinite wherever it is defined.
	Then there exists $Q\in\Mat_n(\R(t))$ with $M=Q^TQ$ if and only if $\det(M)$ is a square in $\R(t)$.
\end{cor}
\begin{proof}
	Clearly, if $M=Q^TQ$ then $\det(M)=(\det Q)^2$ is a square. Conversely, assume that $d:=\det(M)$ is a square. After diagonalization of $M$ (as a quadratic form) we may assume that $M=\langle a_1,\dots,a_n\rangle$ for some $a_i\in\R(t)$. Applying Corollary~\ref{cor:witt_lgp} $n-1$ times, we get
	\[
		M\simeq_{\R(t)}\langle 1,a_1a_2,a_3\dots,a_n\rangle\simeq_{\R(t)}\dots\simeq_{\R(t)}\langle 1,\dots,1,\prod_{i=1}^na_i\rangle.
	\]
	Since $\prod_{i=1}^na_i=d$ is a square, we therefore have $M\simeq\1_n$.
\end{proof}

A classical theorem due to Cassels \citep{Cassels64} (and shortly after generalized by Pfister \citep{Pfister65}) says that for any field $k$ a polynomial $f\in k[t]$ that is a sum of squares of rational functions is already a sum of squares of polynomials (with the same number of squares).
Tignol proved in \citep{Tignol96} a version for univariate polynomial rings over central simple algebras, which, applied to the matrix algebra $\Mat_n(\K)$, gives that any matrix $M\in\Her_n(\K[t])$ that admits a rational factorization $M=Q^*Q$, where $Q\in\Mat_n(\K(t))$, also admits a polynomial factorization $M=P^*P$, where $P\in\Mat_n(\K[t])$.
This statement can also be shown using techniques of Leep's from \cite{Leep06}. We will prove it below, see Corollary~\ref{cor:polynomial_factorization}, as a result of a more elementary and more explicit proof of a Cassels-Pfister Theorem for matrices over $\K[t]$, which we give in the following section.

\section{A Cassels-Pfister Theorem for Matrices}\label{sec:cassels_pfister}
In this section, we present an elementary and explicit proof of the known fact that any matrix $M\in\Her_n(\K[t])$ that admits a rational factorization $M=Q^*Q$, where $Q\in\Mat_n(\K(t))$, also admits a polynomial factorization $M=P^*P$, where $P\in\Mat_n(\K[t])$. Our approach allows us to investigate the pole behavior more closely. The ring $\cO_M$ that we associate to a non-degenerate matrix $\Mat_n(\K[t])$ will play a central role. It consists of those rational functions, that have no poles wherever $M$ becomes degenerate. We prove that, up to equivalence, factorizations of $M$ over $\cO_M$ correspond exactly to factorizations over the polynomial ring, see Theorem~\ref{thm:cassels_pfister}.

\subsection{Splitting-off Matrix Zeros}

To every zero of a (scalar) polynomial corresponds a linear factor that can be split off. Almost the same can be done in the matrix case, if we take care of the order of multiplication in evaluating matrix polynomials: Let $P = \sum_i P_i t^i\in \Mat_n(\K[t])$ be a matrix polynomial and let $A\in \Mat_n(\K)$ be a constant matrix. We fix the notation
\begin{align*}
	P_A:&=\sum_iP_iA^i, \\
	\leftidx{_A}P:&=\sum_iA^iP_i
\end{align*}
for the right and left evaluation of $P$ at $A$, respectively.
\begin{lem}
	For $P$ and $A$ as above, we have
	\[
		P_A=0\Leftrightarrow\exists S\in\Mat_n(\K[t])\colon P=S(t\1_n-A)
	\]
	and
	\[
		\leftidx{_A}P=0\Leftrightarrow\exists S\in\Mat_n(\K[t])\colon P=(t\1_n-A)S.
	\]
\end{lem}

\begin{proof}
	Just as in the scalar case, we use the identity
	\[
		(t^k \1_n -A^k) = (t^{k-1}\1_n + t^{k-2}A + \dots + A^k)(t\1_n -A)
	\]
	to split off $(t\1_n -A)$ from the right of $P-P_A$ which coincides with $P$ if $P_A = 0$. The argument obviously can be adapted to the case $\leftidx{_A}P=0$.
\end{proof}

For a scalar polynomial $q\in \K[t]$ and a complex number $z\in\C$ which is a zero of $q^*q$, we can conclude that $z$ or $z^*$ must be a zero of $q$. The following proposition shows how this can be generalized to matrix polynomials.
\begin{prop}\label{prop:matrix_zeros}
	Let $Q\in\Mat_n(\K[t])$ and $z\in\C$ such that $(Q^*Q)(z)=0$. In case $\K=\R$ we further assume that $n$ is even. Then there exists a constant matrix $A\in \Mat_n(\K)$ with the following properties:
	\begin{itemize}
		\item $A$ is normal and its only eigenvalues are $z$ and $z^*$.
		\item $A$ is a zero of $Q$ under left evaluation, i.e.~$\leftidx{_A}Q=0$. 
	\end{itemize}
	In particular, we can split off a linear factor from $Q$
	\[
		Q=(t\1_n-A)P\quad \text{for some}\quad P\in\Mat_n(\K[t])
	\]
	and we have $(t\1_n-A)^*(t\1_n-A)=(t-z)^*(t-z)\1_n$.
\end{prop}
Before proving this proposition, we give the following simple observation that we need for the case $\K=\R$. Its use in the proof of Proposition~\ref{prop:matrix_zeros} has been inspired by a similar argument in \citep{FRS06}. In order to distinguish from taking the adjoint, we denote the entry-wise complex conjugation of a vector $v\in\C^n$ or a matrix $M\in\Mat_n(\C)$ by $\overline{v}$ and $\overline{M}$, respectively.

\begin{lem}\label{lem:conjugate_complement}
	Let $n$ be even and let $U\subseteq \C^n$ be a subspace that is orthogonal to its conjugate $\overline{U}$ and maximal under inclusion with this property. Then $\dim U=\frac{n}2$. In particular, $\overline{U}$ is the orthogonal complement of $U$.
\end{lem}

\begin{proof}
	The condition $\overline{U}\perp U$ (with respect to the standard hermitian inner product) just means that $U$ is totally isotropic with respect to the bilinear form
	\begin{align*}
		\be\colon \C^n\times\C^n &\to\C\\
		(v,w)&\mapsto v^Tw.
	\end{align*}
	Since $-1$ is a square in $\C$ and $n$ is even, $\be$ is hyperbolic and thus every maximal totally isotropic subspace is of dimension $\frac{n}2$ (see for example \citep[Corollary~I.4.4]{Lam05}).
\end{proof}

\begin{proof}[of Proposition \ref{prop:matrix_zeros}]
	The condition $0=Q^*(z)Q(z)=(Q(z^*))^*Q(z)$ can be read as orthogonality of the images of the linear maps $Q(z)$ and $Q(z^*)$ in $\C^n$, that is,
	\[
		\im Q(z)\perp \im Q(z^*).
	\]
	We are going to choose a subspace $U\subseteq\C^n$ such that
	\[
		\im Q(z)\subseteq U\quad\text{and}\quad U\perp\im Q(z^*).\tag{$\ast$}
	\]
	For $\K=\C$, we just take $U=\im Q(z)$. For $\K=\R$, the appropriate choice of $U$ will ensure that the entries of the constructed matrix lie in $\R$. Rewriting ($\ast$) yields
	\[
		U^{\perp}\subseteq\im Q(z)^{\perp}= \ker Q^*(z^*)
		\text{\quad and \quad} 
		U\subseteq\im Q(z^*)^{\perp}=\ker Q^*(z)
		.\tag{$\ast\ast$}
	\]
	We choose $A$ to be the representing matrix of the operator $z^*\pi_U+z\pi_{U^{\perp}}\in\End(\C^n)$, where $\pi_U$ and $\pi_{U^{\perp}}$ are the orthogonal projections onto $U$ and $U^{\perp}$, respectively. In other words $A^*$ acts on $U$ as multiplication by $z$ and on $U^{\perp}$ as multiplication by $z^*$. Combining this with ($\ast\ast$) we conclude
	\[
		\forall u\in U\colon Q^*_{A^*}u=Q^*(z)u=0
	\]
	as well as
	\[
		\forall w\in U^{\perp}\colon Q^*_{A^*}w=Q^*(z^*)w=0.
	\]
	Since $U$ and $U^{\perp}$ span $\C^n$, this means $Q^*_{A^*}=0$, or equivalently $\leftidx{_A}Q=0$. Moreover, $A$ clearly is normal and its only eigenvalues are $z$ and $z^*$, as desired.

	If $\K=\C$ we are done at this point. So for the rest of the proof we assume $\K=\R$ and $n$ is even. In this case, $\overline{Q}=Q$. Since $\im Q(z)$ is orthogonal to $\im Q(z^*)=\im\overline{Q}(z^*)=\overline{\im Q(z)}$ we also have
	\[
		\im Q(z)\perp\overline{\im Q(z)}.
	\]
	We choose a subspace $U\subseteq\C^n$ containing $\im Q(z)$ and maximal with $U\perp \overline{U}$. Since $\overline{U}\supseteq \overline{\im Q(z)}=\im Q(z^*)$ we also have $U\perp \im Q(z^*)$ as required in ($\ast$). 
	Due to the maximality of $U$ it is the orthogonal complement of its conjugate $\overline{U}$, as observed in Lemma~\ref{lem:conjugate_complement}. In particular $\pi_{\overline{U}}=\pi_{U^{\perp}}$.
	
	It is easily seen that the conjugate of the representing matrix of $\pi_U$ is the representing matrix of $\pi_{\overline{U}}=\pi_{U^{\perp}}$. Using this, it is clear that the matrix $A$ we constructed with the above choice of $U$ has real entries.
\end{proof}

\subsection{Pole Cancellation}
In this subsection, we show how to produce polynomial factorizations from given rational ones using unitary matrices ``capturing'' the poles of the factors, see Theorem~\ref{thm:cassels_pfister}. In the complex case $\K=\C$, our approach is similar to the approach in \citep{ESS16}.

\begin{rem}\label{rem:degreesos}
	In the following, we will often use the simple fact that if $\sum_ia_i^*a_i = 1$ for some polynomials $a_i\in\K[t]$, then all $a_i$ are in fact constant. The reason is that all leading coefficients of the $a_i^*a_i$ are positive and hence the coefficients of the highest degree term cannot cancel each other. In particular, a polynomial unitary matrix $U\in\U_n(\K[t])$ has constant entries, i.e.~$U\in\U_n(\K)$.
\end{rem}

\begin{lem}\label{lem:one_cancellation}
	For any $M\in\Her_n(\K[t])$ and $k\in \N$, the classes of square factorizations of $M$ correspond one-to-one to those of $M\oplus \1_k$. More precisely, the map
	\begin{align*}
		\fac M{\K[t]}{n}&\to\fac {M\oplus\1_k}{\K[t]}{n+k}\\
		[Q]&\mapsto [Q\oplus\1_k]
	\end{align*}
	is a bijection. 
\end{lem}
\begin{proof}
	To show injectivity, let $ U (Q_1 \oplus \1_k) = Q_2 \oplus \1_k$ for some unitary $(n+k)$ matrix. 
	Then $U$ must be of the form $U_1\oplus\1_k$ and we have $U_1Q_1 = Q_2$, that is $[Q_1] = [Q_2]$. 
	To show surjectivity, let $P\in\Mat_{(n+k)}(\K[t])$ such that $P^*P=M\oplus\1_k$. Then the last $k$ columns of $P$ form an orthonormal system and therefore must have constant entries, see Remark~\ref{rem:degreesos}. Extending them to an orthonormal basis of $\K^{n+k}$ shows that there exists a unitary matrix $U\in U_{n+k}(\K)$ which has the same last $k$ columns as $P$. Then $U^*P=Q\oplus \1_k$ for some $Q\in\Mat_n(\K[t])$ with $Q^*Q=M$.
\end{proof}

We use the fact that we can split off linear factors coming from matrix zeros in order to show that we can get rid of poles in (rational) factorizations of polynomial matrices. 

\begin{thm}\label{thm:cassels_pfister_surjective}
	Let $\cO$ be a $*$-invariant subring of $\K(t)$ containing $\K[t]$ and let $S\in \Mat_n(\cO)$. 
	If $S^*S$ has polynomial entries, then there exists a unitary matrix $U\in \U_n(\cO)$ such that $US$ has polynomial entries. In other words, for positive semidefinite $M\in\Her_n(\K[t])$ the canonical map
	\[
		\fac M{\K[t]}n\to\fac M{\cO}n
	\]
	is surjective.
\end{thm}
\begin{proof}
	Note that for every fixed $n$, the two formulations of the theorem are in fact equivalent. For the moment we assume that $n$ is even if $\K=\R$, and keep the odd case for later. First we prove the following intermediate claim:
	
	If $a\in \K[t]$ and $Q\in\Mat_n(\K[t])$ such that $a^*a$ divides the entries of $Q^*Q$ in $\K[t]$, then there exist $U\in \U_n(\K[t]_{a^*})$ such that $UQ\in\Mat_n(\K[t])$ and $a$ divides the entries of $UQ$ in $\K[t]$. For this we may assume that $a$ is irreducible, otherwise we repeat the argument for each irreducible factor of $a$. So let $a$ be monic and irreducible such that $a^*a$ divides the entries of $Q^*Q$. We consider two cases.

	Case 1: $a$ is linear, say $a=t-z$. Since $(Q^*Q)(z)=0$ we can use Proposition~\ref{prop:matrix_zeros} to split off a linear factor $T:=(t-A)$ with $T^*T=a^*a\1_n$ from the left of $Q$, i.e.~$Q=TP$ for some $P\in \Mat_n(\K[t])$. Then $U:=\frac1{a^*} T^*$ is unitary and $UQ=\frac1{a^*}T^*TP=aP$ is divisible by $a$.

	Case 2: $a$ is quadratic, say $a=(t-z)^*(t-z)$ (in particular, $\K=\R$ and $n$ is even). Then $a^2|Q^TQ$. Just as in the first case, we can split off a linear factor $T_0$ from $Q$ with $T_0^TT_0=a\1_n$. So let $Q=T_0P$. Then $a|P^TP$ and we can split off another linear factor $T_1$ from $P$ with $T_1^TT_1=a\1_n$. Then $T:=T_0T_1$ divides $Q$ from the left and $T^TT=a^2\1_n$. In particular, $U:=\frac1aT^T$ is orthogonal and $UQ$ is divisible by $a$. This proves the intermediate claim.

	Now let $S\in\Mat_n(\cO)$ be a matrix of rational functions such that $S^*S$ has polynomial entries. Denote by $a\in\K[t]$ the smallest common denominator of the entries of $S$. Using that $\K[t]$ is a principal ideal domain, it is not hard to see, that $\frac1a\in\cO$. Indeed, given $\frac{c_1}a,\dots,\frac{c_r}a\in\cO$ such that $a$ is coprime to $c_1,\dots,c_r$, then there exists a linear combination $1=\al a+\sum_i\ga_ic_i$ with $\alpha,\ga_1,\dots,\ga_r \in \K[t]$ and therefore 
	\[
		\frac1a=\al+\sum_i\ga_i\frac{c_i}a\in\cO.
	\]
 Since $\cO$ is $*$-invariant, also $\frac1{a^*}\in\cO$. We set $Q:=aS\in\Mat_n(\K[t])$. Since $S^*S$ has polynomial entries, $Q^*Q=a^*aS^*S$ is divisible by $a^*a$. Using the claim, there exists $U\in\U_n(\K[t]_{a^*})\subseteq\U_n(\cO)$ such that $UQ\in\Mat_n(\K[t])$ and $a$ divides the entries of $UQ$ and hence $US=\frac1aUQ$ has polynomial entries, as claimed. This proves the case $\K=\C$ or $n$ even.

	For the remaining case let $\K=\R$ and $n$ be odd. We prove the second formulation of the theorem. So let $M\in\Sym_n(\K[t])$ be positive semidefinite. We look at the following commutative diagram consisting of the canonical maps.
	\begin{center}$
		\begin{tikzcd}
			\fac M{\R[t]}n\ar[r]\ar[d]	& \fac M{\cO}n\ar[d]\\
			\fac {M\oplus \langle 1\rangle}{\R[t]}{n+1}\ar[r]	& \fac {M\oplus \langle 1\rangle}{\cO}{n+1}
		\end{tikzcd}$
	\end{center}
	The left hand map and bottom map are surjective by Lemma~\ref{lem:one_cancellation} and the even case, respectively. Since the right hand map is clearly injective (see the proof of Lemma~\ref{lem:one_cancellation}), this gives the surjectivity of the top map, completing the proof.
\end{proof}

\begin{cor}\label{cor:polynomial_factorization}
	Let $M\in\Sym_n(\R[t])$ be positive semidefinite and non-degenerate. Then there exists $Q\in\Mat_n(\R[t])$ with $M=Q^TQ$ if and only if $\det(M)$ is a square in $\R[t]$.
\end{cor}

\begin{proof}
	Combine Corollary~\ref{cor:rational_factorization} with Theorem~\ref{thm:cassels_pfister_surjective}.
\end{proof}

\begin{cor}\label{cor:polynomial_factorization_nplus1}
	Let $M\in\Sym_n(\R[t])$ be positive semidefinite. Then $M$ admits a factorization $M=Q^TQ$ for some $Q\in\Mat_{(n+1)\times n}(\R[t])$.
\end{cor}
\begin{proof}
	First we reduce to the case that $M$ is non-degenerate, i.e.~has nonzero determinant. Since $\R[t]$ is a principal ideal domain, we can choose a basis of $\ker M$ and extend it to a basis of $\R[t]^n$. The according congruence transformation on $M$ results in a block matrix of the form $M'\oplus 0$, where $M'$ is non-degenerate and positive semidefinite. From a factorization of $M'$ we get one of $M$. Replacing $M$ by $M'$, we may assume that $d:=\det(M)$ is nonzero. Then by the previous Corollary~\ref{cor:polynomial_factorization}, $M\oplus d$ has a square factorization of size $n+1$. The first $n$ columns of the latter give the desired $(n+1)\times n$-factorization of $M$.
\end{proof}

One of the main aims of this paper is not only to prove existence, but to give a precise classification of all such factorizations up to unitary equivalence. For every non-degenerate positive semidefinite matrix $M\in\Her_n(\K[t])$, there exists only one square factorization up to unitary equivalence over $\K(t)$. Namely, let $M=Q^*Q=P^*P$, with $Q,P\in \Mat_n(\K(t))$. Then $Q$ and $P$ only differ by the unitary matrix $U=Q^{-1}P\in\U_n(\K(t))$. That is, over $\K(t)$ all square factorizations of $M$ are equivalent. However, the situation changes if we require the involved unitary matrices to have no poles wherever $M$ is singular, i.e.~to have entries in $\cO_M$. The next proposition shows that any two polynomial factorizations over $\K[t]$ that are equivalent over $\cO_M$, are already equivalent over $\K[t]$. Due to Remark~\ref{rem:degreesos}, they are even equivalent over $\K$.

\begin{prop}\label{prop:cassels_pfister_injective}
	Let $M\in\Her_n(\K[t])$ be positive semidefinite and non-degenerate. Let $\cO$ be a subring of $\cO_M$ containing $\K[t]$. Given $Q_1,Q_2\in\Mat_n(\K[t])$ and $U\in\U_n(\cO)$ such that $M=Q_1^*Q_1=Q_2^*Q_2$ and $UQ_1=Q_2$, then $U\in\U_n(\K[t])$. In particular, the canonical map
	\[
		\fac M{\K[t]}n\to\fac M{\cO}n
	\]
	is injective.
\end{prop}

\begin{proof}
	By assumption, the entries of $U=Q_2Q_1^{-1}$ lie in $\cO\subseteq\cO_M$. By definition of $\cO_M$ this means that potential poles of $U$ can only occur wherever $\det(M)$ does not vanish. But in these points $M$ and thus $Q_1$ are invertible. Therefore, $U$ is defined everywhere and hence polynomial.
\end{proof}

We have now proved the Cassels-Pfister Theorem for matrices over $\K[t]$ that we need for our purposes.
\begin{thm}\label{thm:cassels_pfister}
	Let $M\in\Her_n(\K[t])$ be positive semidefinite and non-degenerate and let $\cO=\cO_M$. Then the canonical map
	\[
		\fac M{\K[t]}n\to\fac M{\cO}n
	\]
	is a bijection.
\end{thm}

\begin{proof}
	Combine Theorem~\ref{thm:cassels_pfister_surjective} and Proposition~\ref{prop:cassels_pfister_injective}.
\end{proof}

\section{The Smith Normal Form of Positive Semidefinite Matrices}\label{sec:snf}

We will see that Theorem~\ref{thm:cassels_pfister} gives us a much stronger result than mere existence of polynomial factorizations as in Corollaries~\ref{cor:polynomial_factorization} and \ref{cor:polynomial_factorization_nplus1}. Namely it allows us to work over the semi-local ring $\cO_M$ instead of $\K[t]$ in order to count the number of square factorizations. The advantage of working over $\cO_M$ lies in the main result of this section, Theorem~\ref{thm:local_snf}, which states that a positive semidefinite matrix $M\in\Sym_n(\K[t])$ is congruent to its Smith normal form if we allow congruence transformations over $\cO_M$.

\medskip
Recall that the \emph{Smith normal form} of a matrix $M\in\Mat_n(R)$ over a principal ideal domain $R$ is a diagonal matrix $D=\langle a_1,\dots,a_k,0,\dots,0\rangle$ ($k$ the rank of $M$) where $a_1,\dots,a_k\in R$ with $a_i|a_{i+1}$ ($i=1,\dots,k-1$) such that there exist invertible matrices $S,T\in\Mat_n(R)$ with $SMT=D$. The $a_i$ are called the \emph{invariant factors} of $M$ and are uniquely determined up to units in $R$. See \citep[Chapter IV]{Macduffee33} for background.
Moreover, if $R=\K[t]$, then requiring the invariant factors to be monic makes them unique and in that case we refer to $\langle a_1,\dots,a_k,0,\dots,0\rangle$ as the \emph{monic Smith normal form} of $M$.

In general the transformation to obtain the Smith normal form of a symmetric matrix $M$ cannot be chosen to be a congruence transformation, i.e.~$S=T^*$ for the above transformation matrices. However, if $R=\K[t]$ and $M$ is positive semidefinite, then this is possible locally ``around'' the roots of $\det(M)$. For a precise statement see Theorem~\ref{thm:local_snf}.

\subsection{Diagonalization over semi-local principal ideal domains.}
First we show that the Gram-Schmidt method for orthogonalization leads to a Smith normal form of a given matrix over many semi-local principal ideal domains.

\begin{lem}\label{lem:snf}
	Let $\cO$ be a semi-local principal ideal domain containing $\Q$. Moreover, let $A\in\Sym_n(\cO)$. Then $A$ is congruent to its Smith normal form. More precisely, there exist $a_1,\dots,a_k\in \cO$, where $k$ is the rank of $A$, with $a_1|a_2|\dots|a_k$ and $A\simeq_{\cO} \langle a_1,\dots,a_k,0,\dots,0\rangle$.
\end{lem}
\begin{proof}
	We proceed by induction on $n$. If $n=0$ there is nothing to prove. So let $n>0$. Then either $A=0$ or the entries of $A$ have a greatest common divisor, denoted by $a_1$. So we have $A=a_1B$ for some $B\in\Sym_n(\cO)$. For each of the finitely many maximal ideals of $\cO$ there is at least one entry of $B$ not contained in it. Denote $q_B\colon (v,w)\mapsto v^TBw$ the bilinear form defined by $B$. By the prime avoidance Lemma~\ref{lem:prime_avoid_quadratic} (Appendix) for quadratic forms, there exists $v\in \cO^n$ such that $q_B(v,v)$ is not contained in any of the maximal ideals and is hence a unit in $\cO$. Since $v$ represents a unit, the submodule $\cO v$ of $\cO^n$ has an orthogonal complement, as can be shown using the Gram-Schmidt orthogonalization method. Restricting the bilinear form $q_B$ to the orthogonal complement and applying the induction hypothesis to any representing matrix, we get the diagonalization as claimed.
\end{proof}
\begin{rem}\label{rem:snf_hermitian}
	Under one additional assumption, the same proof also works for a hermitian matrix $A$ over a semi-local principal ideal domain $\cO\supset\Q$ with involution, using Lemma~\ref{lem:prime_avoid_hermitian}. In the induction step, we need to be able to choose the greatest common divisor $a_1$ of the entries of $A$ to be hermitian, i.e.~$*$-invariant. This is possible, if (and only if) every $*$-invariant ideal (in this case, the ideal generated by the entries of $A$) has a $*$-invariant generator. If $\cO$ is a subring of $\C(t)$ containing $\C[t]$, then this is true, since every ideal is generated by a monic polynomial, which must have real coefficients, if the ideal is $*$-invariant.

	In fact, this condition just means that $\cO$ is unramified over its subring of $*$-invariant elements. To illustrate that this assumption is necessary, we equip $\R[t]$ with the $\R$-linear involution given by $t\mapsto -t$. The ring of $*$-invariant elements is $\R[t^2]$. Clearly, the $*$-invariant ideal $(t)$ has no $*$-invariant representative and it is obviously not possible to diagonalize the hermitian matrix $\begin{pmatrix}0&t\\-t&0\end{pmatrix}$, even over the localization $\cO=\R[t]_{(t)}$.
\end{rem}

\subsection{Avoiding Denominators}
The most technical step in the proof of Theorem~\ref{thm:local_snf} is to keep track of denominators in transforming quadratic forms. Using induction, similarly as in Corollary~\ref{cor:rational_factorization}, we reduce to the case of two dimensional forms and elements represented by them. As is common in quadratic forms theory, it can be quite useful to consider quadratic forms of the form $\langle 1,-c\rangle$. The essential advantage that we are going to exploit is the additional multiplicative structure that we gain by viewing these as norm forms of quadratic ring extensions.

We first give a variant of a standard exercise in number theory about the ring of integers in quadratic number fields.
\begin{lem}\label{lem:quadratic_intergral_closure}
	Let $A$ be a principal ideal domain with field of fractions $K$ and let $c\in A$ be square free. If $2\in A^{\times}$, then $A[\sqrt{c}]$ is the integral closure of $A$ in $K[\sqrt{c}]$.
\end{lem}
\begin{proof}
	Clearly, every element of $A[\sqrt{c}]$ is integral over $K$. Now let $a,b\in K$ and set $x:=a+b\sqrt{c}$. Suppose $x$ is integral over $A$. Then $x^*=a-b\sqrt{c}$ is integral, too. Therefore, $2a=x^*+x\in A$ and thus $a\in A$. Now also $b\sqrt{c}$ is integral. In particular, $b^2c\in A$. Since $c$ is square free, also $b\in A$.
\end{proof}

\begin{lem}\label{lem:ufd_hermitian_square_denominators}
	Let $B$ be a factorial ring with involution. Then for every $a\in L:=\Quot(B)$ with $a^*a\in B$ there exists $b\in B$ such that $b^*b=a^*a$.
\end{lem}

\begin{proof}
	Let $a=\frac{c}d$ with $c,d\in B$ coprime. Then $d$ divides $c^*c$ since $a^*a\in B$. But since $c$ and $d$ are coprime, already $d$ divides $c^*$. In other words $d^*$ divides $c$, i.e.~$a=\frac{bd^*}d$ for some $b\in B$. Clearly $b^*b=a^*a$.
\end{proof}

The following somewhat technical lemma is used to avoid denominators in transformation in the aforementioned two-dimensional forms.
Recall that we write $Z_B(e)$ for the set of all prime ideals of $B$ containing $e$.
\begin{lem}\label{lem:norm_denominators}
	Let $A$ be a semi-local principal ideal domain with field of fractions $K$ and $2\in A^{\times}$. 
	Moreover, let $c,e\in A$ such that $c$ is square free and denote $B:=A[\sqrt{c}]$. Suppose that for all the zeros $\q\in Z_B(e)$ of $e$, the residue field $k(\q)=B/\q$ is quadratically closed\footnote{i.e.~every element is a square.}. Then for every $\ga\in L:= K[\sqrt{c}]$ such that $N_{L|K}(\ga)\in A^{\times}$, there exists $\al \in A[e\sqrt{c}]$ such that $N_{L|K}(\al)=N_{L|K}(\ga)$. 
\end{lem}
\begin{proof}
	Note that since $2$ is a unit, $B$ is the integral closure of $A$ in $L$ by Lemma~\ref{lem:quadratic_intergral_closure} and hence a Dedekind domain, see \cite[Chapter~I, Proposition~(8.1)]{NeukirchMR1697859}. Also, $B$ is clearly semi-local and therefore, a principal ideal domain \cite[Chapter~IV, §4, Exercise 4]{NeukirchMR1697859}. Denote $N:=N_{L|K}$ the norm form of $L|K$, see \cite[Chapter~I, §2]{NeukirchMR1697859}. As an involution $*$ on $L$ and $B$ we fix the nontrivial $K$-automorphism of $L$. Then $N(x)=x^*x$ for all $x\in L$.
	
	Using Lemma~\ref{lem:ufd_hermitian_square_denominators} we may assume that $\ga\in B$, since $\ga^*\ga\in A\subseteq B$.  We are going to construct $\ep\in B^{\times}$ such that $\ep^2\ga\in A[e\sqrt{c}]$.
	Then, for $\al:=\frac{\ep^2}{N(\ep)}\ga$, we have $N(\al)=N(\ga)$ and $\al\in A[e\sqrt{c}]$, as desired, because $N(\ep)\in A^{\times}$.

	To construct such an $\ep$, we use the fact that for all $\q\in Z_B(e)$ the residue field $k(\q)=B/\q$ is quadratically closed as well as $\ga^*(\q)\neq0$ (because $\ga\in B^{\times}$) to conclude by Hensel's Lemma for complete discrete valuations rings (see \cite[Chapter~II, §4, Lemma (4.6)]{NeukirchMR1697859}) that 
	$\ga^*$ is a square modulo any power of $\q$.
	So we can choose a unit $\ep\in B^{\times}$ satisfying the following finitely many congruences for $\q\in Z_B(e)$
	\[
		\ep^2\equiv \ga^* \quad\text{mod }\q^{n_{\q}}
	\]
	where
	\begin{itemize}
		\item $n_{\q}=2v_{\q}(e)+1$ if $\q|A\cap\q$ is ramified (which is the case if and only if $\q$ is a zero of $\sqrt{c}$)
		\item $n_{\q}=v_{\q}(e)$ otherwise.
	\end{itemize}
	Here, $v_{\q}$ denotes the discrete valuation corresponding to $\q$. Let $a,b\in A$ with $\ga\ep^2=a+b\sqrt{c}$. Then we have
	\[
		b=\frac{\ga\ep^2-(\ga\ep^2)^*}{2\sqrt{c}}.
	\]
	By the choice of $\ep$, we get that $v_{\p}(b)\geq v_{\p}(e)$ for all $\p\in Z_A(e)$ and hence $e$ divides $b$. This means $\ga\ep^2\in A[e\sqrt{c}]$ as desired.
\end{proof}

Combining the previous lemma with Witt's local-global principle leads to the following proposition, which is the main step in the proof of Theorem~\ref{thm:local_snf}.
\begin{prop}\label{prop:local_witt}
	Let $a,b,d\in\R[t]$ be nonzero with $a$ and $b$ positive semidefinite and $a|b|d$. Write
	\[
		\cO:=\left\{\, \frac{c}e \mid c,e\in\R[t],\ e\text{ coprime to }d\,\right\}
	\]
	and let $u,v\in\cO^{\times}$ be positive semidefinite. Then
	\[
		\langle au,bv\rangle\simeq_{\cO}\langle a,buv\rangle.
	\]
\end{prop}
\begin{proof}
	Since $a$ divides $b$, we can factor out $a$ and therefore assume that $a=1$. By Corollary~\ref{cor:witt_lgp} to Witt's local-global principle, $1$ is represented by $\langle u,bv\rangle$ over the rational function field $\R(t)$. Dividing by $u$, we get that $\frac1u$ is represented over $\R(t)$ by
	\[
		\left\langle 1,\frac{bv}u\right\rangle\cong_{\R(t)}\langle1,-c\rangle
	\]
	where $c\in\cO$ is the square free part of $-\frac{bv}{u}$, i.e.~$c$ is square free and $-\frac{bv}{u}=e^2c$ for some $e\in\cO$. Being represented by $\langle1,-c\rangle$ means being the norm of an element $\R(t)[\sqrt{c}]$. Clearly, $c$ is negative semidefinite. Since $c$ is square free, it cannot have a real zero. In particular, all quotients of $\cO[\sqrt{c}]$ modulo its maximal ideals are isomorphic to $\C$ and hence quadratically closed. Now we can use Lemma~\ref{lem:norm_denominators} to get a representation of $\frac1u$ as a norm of an element of $\cO[e\sqrt{c}]$. In other words $\frac1u$ is represented by $\langle 1,-e^2c\rangle$ over $\cO$ or (multiplying by $u$) we get that $1$ is represented by
	\[
		\langle u,-ue^2c\rangle=\langle u,bv\rangle.
	\]
	 Hence the latter is congruent to $\langle 1, buv\rangle$ over $\cO$.
\end{proof}

Similarly to Corollary~\ref{cor:rational_factorization} we apply this Proposition inductively to obtain the main result of this section.
\begin{thm}\label{thm:local_snf}
	Let $M\in\Her_n(\K[t])$ positive semidefinite. Then $M$ and its monic Smith normal form are congruent over $\cO_M$.
\end{thm}

\begin{proof}
	Let $\cO:=\cO_M$. Using Lemma~\ref{lem:snf} and Remark~\ref{rem:snf_hermitian} we get $b_1|\dots|b_k\in\cO$ such that
	\[
		M\simeq_{\cO}E:=\langle b_1,\dots,b_k,0,\dots,0\rangle.
	\]
	We may assume that $k=n$, i.e.~$M$ is non-degenerate. Now let $D=\langle a_1,\dots,a_n\rangle$ be the monic Smith normal form of $M$ over $\K[t]$. Then both $D$ and $E$ are Smith normal forms of $M$ over $\cO$. Due to uniqueness there exist units $u_i\in\cO^{\times}$ such that $b_i=a_iu_i$ for all $i$. We now show that $E$ and $D$ are congruent over $\cO$.

	Let $i\in\{1,\dots,n\}$. The determinant of $M$ is divisible by $a_i$. Since $u_i$ is a unit in $\cO$, the numerator and denominator of $u_i$ (in a representation in lowest terms) are coprime to $a_i$ by the definition of $\cO=\cO_M$. In particular, the rational functions $u_i$ and $a_i$ cannot have simultaneous sign changes. Since $b_i=a_iu_i$ is positive semidefinite, both $a_i$ and $u_i$ must be positive semidefinite as well.

	If $\K=\C$, then $u_1,\dots,u_n$ are hermitian squares of units in $\cO$ and we are already done. So now we consider the case $\K=\R$.
	
	We apply Proposition \ref{prop:local_witt} to get that the subform $\langle a_1u_1,a_nu_n\rangle$ of $D$ is congruent to $\langle a_1,a_nu_1u_n\rangle$ over $\cO$. Replacing $u_n$ by $u_1u_n$, we can therefore assume that $u_1=1$. Repeating this argument we can also assume that $u_2=\dots=u_{n-1}=1$. Up to a positive constant, $\prod_{i=1}^na_i$ is the determinant of $M$. Since congruent transformations only change the determinant by a square, we conclude that $u_n$ must be a square and hence can also be assumed to be $1$, which finishes the proof.
\end{proof}

\section{Proof of the Main Theorems}\label{sec:proofs}
As in the previous sections, we let $\K\in\{\R,\C\}$. We combine our previous work in order to prove the main results.

\begin{thm}\label{thm:factorization_snf}
	Let $M\in\Her_n(\K[t])$ be positive semidefinite and non-degenerate. Then the equivalence classes of $n\times n$-factorizations $M=Q^*Q$ are in one-to-one correspondence to those of the monic Smith normal form of $M$. 
\end{thm}
\begin{proof}
	Denote $\cO:=\cO_M$ and let $D$ be the monic Smith normal form of $M$. We consider the following diagram.
	\begin{center}$
		\begin{tikzcd}
			\fac M{\K[t]}n\ar[r, dashed]\ar[d]	& \fac D{\K[t]}n\ar[d]\\
			\fac M{\cO}n\ar[r]			& \fac D{\cO}n
		\end{tikzcd}$
	\end{center}
	By Theorem~\ref{thm:cassels_pfister} we have vertical bijections induced by the inclusion. By Theorem~\ref{thm:local_snf} we have $M\simeq_{\cO}D$, i.e.~there exists an invertible matrix $T\in\Mat_n(\cO)$ with $D=T^*MT$. Right multiplication with $T$ induces a bijection on the bottom.
\end{proof}

From this we obtain the result for $\K=\C$ as mentioned in the introduction.
\begin{cor}\label{cor:complex_factorization}
	Let $M\in\Her_n(\C[t])$ be positive semidefinite with nonzero and square free determinant. Then the determinant map induces a bijection between the sets
	\[
		\{\, Q\in\Mat_n(\C[t]) \mid Q^*Q=M\,\}
	\]
	and
	\[
		\{\, g\in\C[t] \mid g^*g=\det(M)\,\}
	\]
	modulo the left action of the unitary groups $U_n(\C)$ and $U_1(\C)$, respectively.
\end{cor}
\begin{proof}
	Without loss of generality we can assume that $d:=\det(M)$ is monic. Since $d$ is square free, the monic Smith normal form of $M$ is given by 
	\[
		D=\langle 1,\dots,1,d\rangle=\1_{n-1}\oplus \langle d\rangle.
	\]
	We combine the two bijections
	\[
		\fac M{\C[t]}n\to\fac D{\C[t]}n=\fac {\1_{n-1}\oplus\langle d\rangle}{\C[t]}n\to \fac {\langle d\rangle}{\C[t]}1
	\]
	from Theorem~\ref{thm:factorization_snf}  and Lemma~\ref{lem:one_cancellation}, respectively. By Remark~\ref{rem:degreesos}, $\U_n(\C[t])=\U_n(\C)$. Following the construction, it is clear that  the map is induced by the determinant.
\end{proof}

\subsection{Cauchy-Binet}
We want to show that the analogue of Theorem~\ref{thm:factorization_snf} holds also for $(n+1)\times n$ instead of $n\times n$-factorizations. For this we observe the following. For any integral domain $R$ and nonzero $d\in R$, the decompositions
\[
	d=a^2+b^2
\]
as a sum of two squares are basically the same as $2\times 2$-factorizations
\[
	\langle d, d\rangle=Q^TQ
\]
since for $d=a^2+b^2$ we can extend the vector $(a,b)^T$ to the matrix $Q=\begin{pmatrix}a&-b\\b&a\end{pmatrix}$. This observation can also be generalized to higher dimensions in the following way.
\begin{lem}\label{lem:cauchy_binet}
	Let $R$ be an integral domain, $M\in\Sym_n(R)$ with nonzero determinant $d=\det(M)$ and $Q\in \Mat_{(n+1)\times n}(R)$ such that $Q^TQ=M$. Then $Q$ can be extended to a square factorization of $M\oplus \langle d\rangle$. More precisely, there exists a vector $v\in R^{n+1}$ such that for $P=(Q|v)\in \Mat_{n+1}(R)$ we have $P^TP=M\oplus \langle d\rangle$. Moreover, $v$ is uniquely determined up to a scalar factor $\pm 1$.
\end{lem}
\begin{proof}
	Uniqueness is clear, since the vector space (over $\Quot R$) of solutions to $Q^Tv=0$ is one-dimensional. So there are at most $2$ solutions with the additional requirement $v^Tv=d$. To show existence, we define the $i$-th component of $v$ to be the $i$-th maximal minor of $Q$ with sign $(-1)^i$. Then by Cramer's rule $Q^Tv=0$. Moreover, using the Cauchy-Binet formula to compute the determinant of $Q^TQ$, we get that $v^Tv=\det(Q^TQ)=d$.
\end{proof}

\begin{rem}\label{rem:cauchy_binet}
	We want to note that the analogue of the preceding lemma holds for factorizations of hermitian matrices over rings with involutions. The difference is that the extending vector $v$ is determined up to a unit of norm $1$, i.e.~a factor $u\in R^{\times}$ with $u^*u=1$ (instead of $u^2=1$).
\end{rem}

With $R$ and $M$ as above (using Remark~\ref{rem:cauchy_binet}, $M$ might as well be hermitian, if $R$ is a ring with involution), this lemma can be essentially reformulated as follows.
\begin{cor}\label{cor:cauchy_binet}
	Removal of the last column induces a bijection
	\[
	\pushQED{\qed}
		\fac {M\oplus \langle \det(M)\rangle}R{n+1}\to\fac MR{n+1}. \qedhere
	\]
\end{cor}

Using this observation, we get the corresponding result of Theorem~\ref{thm:factorization_snf} for $(n+1)\times n$-factorizations.
\begin{cor}\label{cor:factorization_snf}
	Let $M\in\Her_n(\K[t])$ be positive semidefinite and non-degenerate. Then the equivalence classes of $(n+1)\times n$-factorizations $M=Q^*Q$ are in one-to-one correspondence to those of the monic Smith normal form of $M$. 
\end{cor}
\begin{proof}
	We may assume, that $d:=\det(M)$ is monic. If $D$ is the monic Smith normal form of $M$, then $D\oplus \langle d\rangle$ is the monic Smith normal form of $M\oplus\langle d\rangle$. So we can complete the following diagram to make it commute.
	\begin{center}$
		\begin{tikzcd}
			\fac {M\oplus \langle d\rangle}{\K[t]}{n+1}\ar[r]\ar[d]	& \fac {D\oplus \langle d\rangle}{\K[t]}{n+1}\ar[d]\\
			\fac M{\K[t]}{n+1}\ar[r, dashed]			& \fac D{\K[t]}{n+1}
		\end{tikzcd}$
	\end{center}
	We have a bijection on top by Theorem~\ref{thm:factorization_snf} and vertical bijections by Corollary~\ref{cor:cauchy_binet}.
\end{proof}

Now we can prove the case of particular interest over the field of real numbers.
\begin{cor}\label{cor:real_factorization}
	Let $M\in\Sym_n(\R[t])$ be positive semidefinite with nonzero and square free determinant $\det(M)$. Then there is a bijection between the sets
	\[
		\{\, Q\in\Mat_{(n+1)\times n}(\R[t]) \mid Q^TQ=M\,\}
	\]
	and
	\[
		\{\, g\in\R[t]^2 \mid g^Tg=\det(M)\,\}
	\]
	modulo the left action of the orthogonal groups $\O_n(\R)$ and $\O_2(\R)$, respectively.
\end{cor}
\begin{proof}
       The proof works just as for Corollary~\ref{cor:complex_factorization}. Only we use Corollary~\ref{cor:factorization_snf} instead of Theorem~\ref{thm:factorization_snf} in order to obtain a bijection
       \[
	       \fac M{\R[t]}{n+1}\to\fac {\langle d\rangle}{\R[t]}2.
       \]
\end{proof}

\begin{rem}\label{rem:real_factorization_explicit}
	While it is clear in the complex case, Corollary~\ref{cor:complex_factorization}, that the map in question is given by the determinant, we do not have such an obvious description in the real case.

	Given a factorization $M=Q^TQ$, where $Q\in\Mat_{n+1}(\R[t])$, we follow the construction steps to get a representation of $d:=\det(M)$ as a sum of two squares. By Cauchy-Binet, we get a representation of $d$ as a sum of $n+1$-squares, as in Lemma~\ref{lem:cauchy_binet}. More precisely, there exists $v\in\R[t]^{n+1}$ such that for $Q_1:=(Q|v)\in\Mat_{n+1}(\R[t])$ we have $M\oplus \langle d\rangle=Q_1^TQ_1$. By Theorem~\ref{thm:local_snf} there exists an invertible matrix $T\in\Mat_n(\cO_M)$ such that $T^TMT=\1_{n-1}\oplus d$. For $Q_2:=(QT,v)$ we therefore get $\1_{n-1}\oplus\langle d,d\rangle=Q_2^TQ_2$. By Theorem~\ref{thm:cassels_pfister} there exists $U\in\O_{n+1}(\cO_M)$ such that $UQ_2$ is polynomial and by further applying Lemma~\ref{lem:one_cancellation} we may assume that $UQ_2=\1_{n-1}\oplus A$, for some $A\in\Mat_2(\R[t])$ with $\langle d,d\rangle= A^TA$. Therefore, the last column of $UQ_2$ is given by 
	$Uv=(0,\dots,a,b)^T$, 
	 where $g=(a,b)^T\in\R[t]^2$ is the last column of $A$, which is the desired $2\times 1$-factorization of $d$.

	In short, for the representation $v\in\R[t]^{n+1}$ of $\det(M)$ as a sum of $n+1$-squares coming from the application of Cauchy-Binet to compute $\det(Q^TQ)$, there exists an \emph{appropriate} orthogonal matrix $U\in\O_{n+1}(\cO_M)$, such that $Uv$ is essentially a vector $g$ of length two.

	The first natural question that arises from this observation is whether $v$ can already be compressed over $\R$ to a vector $g$ of length two, that is, can we choose the above matrix $U$ to have entries in $\R$? We want to give an example that this is generally not the case, i.e.~the denominators in $U$ are really necessary:

	Pick any $a,b\in\R[t]$ such that $1,a,b$ are $\R$-linearly independent (e.g. $a=x$, $b=x^2$). We set
	\[
		Q:=\begin{pmatrix}1&0\\0&1\\a&b\end{pmatrix}\in \Mat_{3\times 2}(\R[t])
	\]
	which is a factorization of
	\[
		M:=Q^TQ=\begin{pmatrix}a^2+1 & ab\\ab & b^2+1\end{pmatrix}.
	\]
	The vector $v=(-a,-b,1)^T$ of maximal minors of $Q$ gives a representation $a^2+b^2+1=v^Tv$ of the determinant of $M$ as a sum of three squares. By the above, $v$ is equivalent to a vector of length two, over $\cO_M$, not however over $\R$ by the choice of $a,b$.

	The next question now is, whether any $U\in\O_{n+1}(\cO_M)$ for which $Uv$ is of length two does the job.
	More generally, there is the following open problem.
\end{rem}

\begin{quest*}
	Given any $g,h\in\R[t]^2$ with $g^Tg=h^Th=:d$ such that $g$ and $h$ are $\O_n(\cO_{\langle d\rangle})$-equivalent (after appending $n-2$ zeros), are $g$ and $h$ already $\O_2(\R)$-equivalent?
\end{quest*}

\section{Appendix: Some Prime Avoidance}\label{sec:appendix}

The following is a polynomial version of the prime avoidance lemma. Since we could not find a reference, we include it here.
\begin{lem}\label{lem:prime_avoid_polynomial}
	Let $R$ be a commutative ring and $\p_1,\dots,\p_n\in\Spec R$ such that all quotients $R/\p_i$ are infinite. Moreover let $f\in R[x_1,\dots,x_s]$ be a polynomial that does not lie in $\bigcup_{i=1}^n\p_iR[x_1,\dots,x_s]$. Then there exists $a\in R^s$ such that $f(a)\notin\bigcup_{i=1}^n\p_i$.
\end{lem}

\begin{proof}
	We may assume that the $\p_i$ are pairwise incomparable with respect to inclusion. Passing over to the quotient modulo $I:=\bigcap_i \p_i$ we may further assume that $I=0$. We define the multiplicative set $S:=R\setminus \bigcup_i\p_i$, the localization $Q:=S^{-1}R$, and the prime ideals $\q_i:=\p_iQ$. Any element that is not contained in $\bigcup_i\q_i$ is invertible. Using the Prime Avoidance Lemma \citep[Lemma~3.3]{Eisenbud95} we thus conclude, that the $\q_i$ are maximal ideals. In particular, they are pairwise coprime. By the Chinese Remainder Theorem, $Q$ is isomorphic to the product of the fields $K_i:=Q/\q_i$. Via this isomorphism we identify elements $q\in Q$ with tuples $(q_1,\dots,q_r)\in \prod_{i=1}^rK_i$ and refer to $q_1,\dots,q_r$ as the components of $q$. Likewise we identify $f\in Q[\underline{x}]$ with a tuple of polynomials
	\[
		f=(f_1,\dots,f_r)\in\prod_{i=1}^rK_i[\underline{x}]=Q[\underline{x}]
	\]
	and by assumption none of its components are zero. Since the $K_i$ are infinite there exists $b_i\in K_i^s$ with $f_i(b_i)\neq0$ for every $i$. Again we identify the tuple $(b_1,\dots,b_r)\in \prod_{i=1}^rK_i^s$ with an element $b\in Q^s$. Denote by $c\in S$ a common denominator of all components of $b$, i.e.~$c b\in R^s$. 
	None of the components $g_i:=f_i(tb_i)$ of the univariate polynomial $g:=f(tb)\in Q[t]$ is the zero polynomial, since $g_i(1)\neq 0$. By possibly replacing $c$ by a suitable power of $c$, we may assume that every component $c_i$ of $c$ 
	is either $1$ or not a root of unity. Then for each $i$ the set $\{\, c_i^d \mid d\in\N\,\}$ is either infinite or $\{1\}$. We can therefore choose $0<d\in\N$ such that for all $i$ we have $g_i(c_i^d)\neq0$. Now $a:=c^db$ lies in $R^s$ and has the property that $f(a)=g(c^d)$ does not lie in any of the $\p_i$, as desired.
\end{proof}

\begin{lem}\label{lem:prime_avoid_quadratic}
	Let $R$ be a commutative ring containing $\Q$ and let $\p_1,\dots,\p_r\in\Spec R$. Suppose we have $A\in\Sym_n(R)$  such that $A\notin \bigcup_i\p_i^{n\times n}$. Then there exists $v\in R^n$ with
	\[
		q_A(v,v)=v^TAv\notin\bigcup_i\p_i.
	\]
\end{lem}
\begin{proof}
	Since $2\in A^{\times}$ and $A\notin \bigcup_i\p_i^{n\times n}$ we conclude that the polynomial $x^TAx\in R[x_1,\dots,x_n]$ does not lie in $\bigcup_i \p_i R[x_1,\dots,x_n]$ (use the polarization identity for quadratic forms). The claim now follows from Lemma~\ref{lem:prime_avoid_polynomial}.
\end{proof}

For the sake of completeness, we also provide the following generalization for hermitian forms.
\begin{lem}\label{lem:prime_avoid_hermitian}
	Let $R\supset \Q$ be a commutative ring with involution and let $\p_1,\dots,\p_r\in \Spec R$. Suppose we have $M\in\Her_n(R)$  such that $M\notin \bigcup_i\p_i^{n\times n}$. Then there exists $v\in R^n$ with
	\[
		h_M(v,v)=v^*Mv\notin\bigcup_i\p_i.
	\]
\end{lem}
\begin{proof}
	Denote by $S$ the ring of elements which are fixed by the involution on $R$. In order to make the proof more transparent, we first consider the case that there exists a skew-hermitian unit, i.e.~$u\in R^{\times}$ with $u^*=-u$. In this case $R=S+uS$, since $2\in R^{\times}$. We write $M=A+uB$, with $A,B\in \Mat_n(S)$, $A$ symmetric and $B$ skew-symmetric, thinking of $A$ as the real part and $B$ as the imaginary part of $M$.  

	Now we fix $i\in \{1,\dots,r\}$ for a moment and consider $\p=\p_i$. If $A\notin \p^{n\times n}$, we find $z\in S^n$ such that $h_M(z,z)=z^*Mz=z^TAz\notin\p$ using the polarization identity. Otherwise, $B\notin\p^{n\times n}$ and we take $z=e_i+ue_j$ if the $ij$-th entry of $B$ does not lie in $\p$. Then again $h_M(z,z)\notin \p$. We repeat this for every $i$ and apply Lemma~\ref{lem:prime_avoid_polynomial} to the polynomial map $S^n\times S^n\to S, (x,y)\to h_M(x+uy,x+uy)$ to get $v$ such that $h_M(v,v)\notin \p$.

	In case there exists no such skew-hermitian unit $u\in R^{\times}$, we still can write $M=A+B$, where $A\in \Sym_n(S)$ and $B\in\Mat_n(R)$ is skew-symmetric with skew-hermitian entries. Then for every $i\in\{1,\dots,r\}$, we choose skew-hermitian $u_i\in R$ such that
	\begin{itemize}
		\item $u_i=0$, if $A\notin \p_i^{n\times n}$,
		\item any $u_i\notin \p_i$, if $A\in \p_i^{n\times n}$ and hence $B\notin \p_i^{n\times n}$.
	\end{itemize}
	Similarly as above, it is easy to check that for every $i$ there exist $z_i\in R^n$ of the form $z_i=x_i+u_i y_i$ ($x_i,y_i\in S^n$) such that $h_M(z_i,z_i)\notin \p_i$. Now we consider the map
	\begin{align*}
		\ph\colon \left(S^n\right)^{r+1}&\to S^n\\
		w=\left(w_0,\dots,w_r\right)&\mapsto w_0+\sum u_iw_i.
	\end{align*}
	Again, applying Lemma~\ref{lem:prime_avoid_polynomial} to the polynomial map $w\mapsto h_M(\ph(w),\ph(w))$, proves the existence of $v=\ph(w)$ such that $h_M(v,v)\notin \bigcup \p_i$.
\end{proof}

\bibliographystyle{alpha}
\bibliography{lit}

\end{document}